\documentclass{article}
\usepackage{amssymb,amsfonts,amsmath,amsthm,amscd}
\usepackage{makeidx}
\usepackage[english]{babel}
\makeindex
\usepackage{tikz-cd}  
\usepackage{graphicx}
\usepackage{multicol}
\usepackage{tikz}

\usepackage{hyperref}
\input xy
\xyoption{all}
\usepackage[all]{xy}

\usepackage{fancyhdr}
\usepackage{makeidx}
\makeindex 
\usepackage{vmargin} 
\setmargins{3.3cm}       
{3cm}                        
{14.5cm}                      
{23cm}                    
{0.1pt}                           
{0.5cm}                           
{0pt}                             
{1cm}                           

\numberwithin{equation}{section}

\newtheorem{teo}{Theorem}[section]
\newtheorem{pro}[teo]{Proposition}
\theoremstyle{definition}

\newtheorem{lem}[teo]{Lemma}
\newtheorem{ejem}[teo]{Example}

\newtheorem{coro}[teo]{Corollary}

\newcommand{\m}{{}^{-1}}

\title{ \textbf{ Twisted products: Enveloping actions  and equivariant absolute neighborhood extensors}
}
\author{
	Luis Mart\'inez\footnote{This author was supported by CONAHCYT (México)} \\
	\small Departamento de Matem\'{a}ticas, Facultad de Ciencias, Universidad Nacional Aut\'onoma de M\'{e}xico \\
	H\'{e}ctor Pinedo \\
	\small Escuela de Matem\'{a}ticas, Facultad de Ciencias, Universidad Industrial de Santander \\
	\small  e-mail: luchomartinez9816@hotmail.com, hpinedot@uis.edu.co\\
}

\date{\today}

\begin{document}

	\maketitle
	\begin{abstract} 
		The classical notion of twisted product is studied in the context of partial actions, in particular, we show that the globalization of a partial action is a twisted product. In addition, we establish conditions for the metrizability of twisted products, and some homotopy and categorical properties are proved. Furthermore, sufficient conditions for the enveloping space to be an equivariant absolute neighborhood extensor are also studied.
	\end{abstract}
 \section{Introduction} 
 The concept of partial action appeared in the works \cite{R} and \cite{MK} as a useful tool to study   $C^*$-algebras, later
 advances in the study of topological and continuous partial actions permitted subsequent, algebraic,  $*$-algebraic and topological developments, (see for instance  \cite{ RGG, L, PU1} and the references therein),  and  for a detailed account of recent developments around partial actions the interested reader may consult \cite{D} or \cite{RB}. In this work we are  interested in partial actions of  topological groups on topological  spaces, these appear naturally by restricting usual (global) actions to open sets (see \eqref{induced}). This construction of partial actions suggests the natural problem of whether a given partial action of a group $G$ on a space $X,$ it can be obtained as a  restriction of a global action on some superspace, such an action is the so-called enveloping action or globalization, provided that certain minimality condition is satisfied which guarantees its uniqueness.  The answer to this problem turned out to be affirmative, as shown in \cite{AB} and independently in \cite{KL}, nevertheless topological properties of the initial space are not necessarily inherited by the  space endowed with a global action. For instance, by \cite[Example 1.4]{AB} there is a partial group action on a Hausdorff space whose (minimal) globalization acts on a non-Hausdorff space. In contrast, in \cite[Proposition 1.2]{AB} and \cite[Proposition 2.1]{AB} a criterion was given for the preservation  of the Hausdorff property under globalization. Further, criteria for the preservation of other structural properties such as Polishness were presented in\cite{PU1}, while, regularity, metrizability, the second countable property, having invariant
metric and to be profinite were presented in the works \cite{MPV1} and \cite{MPV2}, respectively.\noindent

On the other hand, the equivariant theory of retracts (see \cite{A2} for a detailed introduction to this subject) appears as a part of the theory of topological transformation groups. In this theory, $G$-spaces are studied from the point of view of the theory of retracts  \cite{SH}. The notions of equivariant absolute (neighborhood) retracts and equivariant absolute (neighborhood) extensors constitute important objects of this theory. Topics related to characterization of equivariant absolute neighborhood retracts  are a very active area  in equivariant infinite-dimensional topology, for instance in  the work \cite{A},  the author presents some open problems related to the topology of the Banach-Mazur compacta, actions on hyperspaces, and invariant metric problems. When $G$ is a compact Lie group, there are characterizations of equivariant absolute (neighborhood) extensors in terms of fixed points (see \cite{JJ} and \cite{JJ2}). The importance of $G$ being a compact Lie group relies on the Slice Theorem (see \cite[Chapter II, Theorem 5.4]{GB}). The fundamental notion of a slice connects us directly with that of a twisted product (see \cite[Chapter II, Theorem 4.2]{GB}), and also with the topics that we will develop in this work.

  This article is structured as follows. After the introduction in the first part of  Section \ref{pre} we provide the necessary notations, notions, and results on partial actions which will be useful throughout the text, especially we  recall the construction of categories $\mathcal{PA}_G$  and $\mathcal{A}_G$ of continuous partial and global actions of a topological group $G$ on topological spaces, respectively. A useful construction in $\mathcal{PA}_G$ are finite products, these  are described in \eqref{tre} and \eqref{treta},  moreover  in Proposition \ref{glof} we recall that there exists a functor, called the enveloping functor, $E: \mathcal{PA}_G \rightarrow \mathcal{A}_G$  sending a $G$-space $(X,\theta),$ to its enveloping space $(X_G, \mu^\theta)$ which turns out to be a left adjoint to the inclusion functor   $\mathcal{A}_G \rightarrow \mathcal{PA}_G.$  Next,  in Section \ref{tp} we explore the classical notion of twisted products (see \cite{GB} and \cite{TD}) in the context of partial action. In Proposition \ref{globprod} is presented their relation with orbit equivalence spaces, moreover structural properties of twisted products and enveloping spaces are presented in Proposition \ref{iota2} and Proposition \ref{iota3},  and in Theorem \ref{glof2}   we construct a pair of adjoint functors between the categories  $\mathcal{A}_G $ and $ \mathcal{PA}_K,$ where $K$ is a subgroup of $G,$ this result recovers   Proposition \ref{glof} in the case $K=G,$ we finish Section \ref{tp} with Theorem   \ref{mett}, where we provide a necessary condition for the metrizability of twisted products. Later, the beginning of Section \ref{gc} is devoted to presenting the required concepts around ANE,  $G$-ANE  spaces, contractible and locally contractible spaces, and relations among them. Further, we show that 
  relevant structural properties are preserved by the enveloping functor $E,$ for instance in Proposition \ref{lcn} we prove that  properties such as locally contractible, LC$^n$ are preserved by $E,$ while Proposition \ref{Ghomot} shows that this functor also preserves homotopies. Then, we turn our attention to  the notions of $G$-contractibility and local $G$-contractibility and show in Theorem \ref{G-con} that the $G$-contractibility  is also preserved by $E,$ further under the assumption that $G$ is compact, one shows in Theorem \ref{locallyc} that $X$ is locally $G$-contractible if and only if $X_G$ is.   Finally in Section \ref{gane}, we work mainly with partial actions on metrizable spaces such that their enveloping spaces are also metrizable, under these assumptions, we show in Theorem \ref{comp}  that for $G$ a  compact Lie group, the space $X$ is locally contractible, if and only if, its enveloping space is a $G$-ANE.  At this point, it is important to observe that the theory of compact Lie transformation groups is a large and well-researched area of mathematics, in this sense we present in Corollary   \ref{cl} that conditions of being an ANE for $X$ and $X_G$ are equivalent provided that $X$ is endowed with a free partial action of a compact Lie group. We finish the work with Theorem \ref{automatic} where we provide a sufficient condition for the space $X_G$ to be a $G$-ANE, where under the assumption that the space $X$ is endowed with a partial action with clopen domains of a finite group $G.$ 
  Several examples are presented throughout the work to illustrate the results.\\
  {\bf Notations:} Given a category $\mathcal C,$ its class of objects is denoted by $\mathcal C_0$, and given $X,Y\in \mathcal C_0,$ the set of morphisms from $X$ to $Y$ is denoted by $\mathcal C(X,Y).$ All functors are covariant. 
  \section{Partial actions and globalization}\label{pre}
In this section, we present some notions and results that will be helpful throughout the work.
Let $G$ be a group with identity element $e$ and $X$ be a set. Following \cite{KL} a   partially defined  function $\theta: G*X\subseteq G\times X\to  X$, $(g,x)\mapsto g\cdot x$,  is a   (set theoretic)  {\it partial action} of $G$ on $X$ if for each $g,h\in G$ and $x\in X$ the following assertions hold:
	\begin{enumerate}
		\item [(PA1)] If $\exists\ g\cdot x$, then $\exists\ g^{-1}\cdot (g\cdot x)$ and $g^{-1}\cdot (g\cdot x)=x$,
		\item [(PA2)] If $\exists\ g\cdot(h\cdot x)$, then $\exists\ (gh)\cdot x$ and $g\cdot(h\cdot x)=(gh)\cdot x$,
		\item [(PA3)] $\exists\ e\cdot x$ and $e\cdot x=x,$
	\end{enumerate}
where $\exists\ g\cdot x$ means that the pair $(g,x)$ belongs to $G*X.$ We say that $\theta$ is a set theoretic {\it global} action if $\exists\ g\cdot x,$ for all $(g,x)\in G\times X.$ For each $g\in G$ consider $X_g=\{x\in X: \exists\ g\m\cdot x\}$. Then $\theta$ induces a family of bijections $\{\theta_g\colon X_{g\m}\ni x\mapsto g\cdot x\in X_g \}_{g\in G}.$   We also denote this family by $\theta.$ Notice that $G*X=\{(g,x)\in G\times X\mid x\in X_{g\m}\},$ moreover $\theta$  acts globally on $X$ if and only if $X_g=X,$ for any $g\in G,$ or equivalently $G*X=G\times X.$ 

Let $G$ be a Hausdorff topological group, $X$ be a topological space, and $\theta: G*X\rightarrow X$ be a set-theoretic partial action of $G$ on the underlying set $X$, endow $G\times X$ with the product topology and $G*X$  with the topology of subspace. Then $\theta$ is called:
\begin{enumerate}
    \item [(a)] a {\it topological partial action} provided that $X_g$ is open in $X$ and $\theta_g: X_{g\m}\rightarrow X_g$ is homeomorphism, for each $g\in G$;
    \item [(b)] a {\it continuous partial action} if it is continuous;
    \item [(c)] a {\it nice partial action} provided that $\theta$ is continuous and $G*X$ is open in $G\times X$.
\end{enumerate}

\vspace{0.2cm}

Let $\mu \colon G\times Z\to Z$ be a continuous global action of $G$ on a  topological space $Z$ and let $X$ be a nonempty open subset of $Z$.  A nice partial action of $G$ on $X$ is obtained by setting:
\begin{equation}\label{induced}X_g=X\cap \mu_g(X),\,\,\, \,\, \theta_g=\mu_g\restriction X_{g\m}\,\,\,  \text{ and }\, \,\,\,\theta \colon G* X\ni (g,x)\mapsto \theta_g(x)\in X .\end{equation}  We say that   $\theta$   is  the {\it restriction of $\mu$ to X} or that $\theta$ is {\it induced} by $\mu$.

A classical problem in the study of partial actions is whether they can be restrictions of global actions. In the topological sense, this turns out to be affirmative and a proof was given in \cite[Theorem 1.1]{AB} and independently in \cite[Section 3.1]{KL}.  For the reader's convenience, we recall their construction. 

Let $\theta$ be a topological partial action of $G$ on a topological space $X$ and consider the following equivalence relation on $ G\times X$:
\begin{equation}
\label{eqgl}
(g,x)R(h,y) \Longleftrightarrow x\in X_{g\m h}\,\,\,\, \text{and}\, \,\,\, \theta_{h\m g}(x)=y.
\end{equation}
 The equivalence class of  $(g,x)$ is denoted by  $[g,x]$. The {\it enveloping space}  or the {\it globalization}  of $X$ is the set  $X_G=(G\times X)/R$  endowed with the quotient topology.  By  \cite[(iii) Proposition 3.9]{KL}  the map
\begin{equation}
\label{action}
\mu^\theta \colon G\times X_G\ni (g,[h,x])\to [gh,x]\in X_G,
\end{equation}
is a continuous action of $G$ on $X_G$, which is called the {\it enveloping action} of $\theta.$ 
The quotient map
\begin{equation} \label{qo} p^\theta: G\times X\ni (g,x)\mapsto [g,x]\in X_G,
\end{equation}
is open by \cite[(ii) Proposition 3.9]{KL}.  On the other hand, by \cite[(i) Proposition 3.9]{KL} the map  
\begin{equation}
\label{iota}
\iota^\theta \colon X\ni x\mapsto [e,x]\in X_G
\end{equation} 
 is continuous, injective and $G\cdot \iota(X)=X_G$. When there is no chance for confusion, we will drop $\theta$ in (\ref{action}), (\ref{qo}) and (\ref{iota}).

 The following result was shown in \cite[Proposition 3.11]{KL} and \cite[Proposition 3.12]{KL}.
\begin{pro}\label{Kellen1}
    \rm{Let $\theta$ be a topological partial action of $G$ on a topological space $X$. Then the following assertions hold:
    \begin{enumerate}
        \item [(i)] $\iota(X)$ is open in $X_G$ if and only if $G*X$ is open in $G\times X$;
        \item [(ii)] $\iota: X\rightarrow \iota(X)$ is an open map if and only if $\theta$ is continuous;
        \item [(iii)] $\iota: X\rightarrow X_G$ is an open embedding if and only if $\theta$ is nice.
    \end{enumerate}}
\end{pro}

Throughout this paper all partial actions will be topological and $G$ will denote a Hausdorff topological group. By a $G$-space we mean a pair $(X,\theta)$, where $X$ is a topological space and $\theta$ is a continuous partial action of $G$ on $X$. The partial action $\theta$ is called \textit{free} provided that $\theta(g,x)\neq x$ for each $(g,x)\in G*X$ such that $g\neq e$.

Let $(X,\theta)$ and $(Y,\eta)$ be $G$-spaces. A continuous map $f: X\rightarrow Y$ is called a {\it $G$-map} when $(g,f(x))\in G*Y$ and $\eta(g,f(x))=f(\theta(g,x))$, for each $(g,x)\in G*X$. Moreover, $f$ is called a {\it $G$-homeomorphism} provided that $f$ is a homeomorphism and $f^{-1}$ is a $G$-map, in that case  the spaces $X$ and $Y$ are called $G$-homeomorphic. Generally, a  $G$-map which is a homeomorphism may fail
to be a $G$-homeomorphism. A $G$-map $f$ is called $G$-{\it isovariant} provided that $G_x=G_{f(x)}$ for each $x\in X$, where $G_x=\{g\in G_x: \theta(g,x)=x\}$ is the {\it isotropy} group of $x$.  Observe that $G_x\subseteq G_{f(x)}$ for each $x\in X$.

We denote by $\mathcal{PA}_G$ the category whose objects are $G$-spaces and whose morphisms are $G$-maps. The composition of morphisms is the composition of functions. Also, we denote by $\mathcal{A}_G$ the full subcategory of $\mathcal{PA}_G$ whose objects are $G$-spaces equipped with global actions of $G$.

Let $J$ be a finite set and $\{(X_j,\theta^j)\}_{j\in J}$ be a collection of $G$-spaces. Put $X=\prod\limits_{j\in J}X_j$ with the product topology and set
	\begin{equation}\label{tre}
	    G*X=\{(g,(x_j)_{j\in J})\in G\times X: (g,x_j)\in G*X_j,\  \forall j\in J \}.
	\end{equation}
	We define the diagonal partial action:
	\begin{equation}\label{treta}
		\Delta(\theta):G*X\ni  (g,x)\mapsto (\theta^j(g,x_j))_{j\in J} \in  X,
	\end{equation}
	where $x=(x_j)_{j\in J}\in X$. Then  $\Delta(\theta)$ is a continuous partial action of $G$ on $X$, and $G*X$ is open (closed) in $G\times X$ provided that $G*X_j$ is open (closed) in $G\times X_j$, for each $j\in J$. Moreover, for any $j\in J$ the projection $\rho_j: X\rightarrow X_j$ is a $G$-map and the triple $(X,\Delta(\theta),\{\rho_j\}_{j\in J})$ satisfies the universal property of the product in $\mathcal{PA}_G$ (see \cite[Lemma 2.7]{MPV2}).

Now we recall the next result. 
 \begin{pro}\label{isog}{\rm \cite[(ii) Proposition 2.11]{MPV1} Let $\beta$ be a continuous global action of $G$ on a space $Y$, and let $\theta$ be the  partial action obtained by restricting $\beta$ to an open subset $X \subseteq Y.$  If $Y = \bigcup_{g\in G}\beta_g(X)$, then $Y$  and $X_G$ are $G$-homeomorphic.}
\end{pro}

The construction of the enveloping space induces a functor $E: \mathcal{PA}_G\rightarrow \mathcal{A}_G$ called the  {\it globalization} (or {\it enveloping functor}),  defined as follows. Given $(X,\theta)$ in $(\mathcal{PA}_G)_0$, one sets  $E(X,\theta)=(X_G,\mu)$, where $\mu$ is given by (\ref{action}).  Now, for $(X,\theta)$ and $(Y,\eta)$ in $(\mathcal{PA}_G)_0$, and a $G$-map $f:X\rightarrow Y$, the functor $E$ sends $f$ to $E(f): X_G\ni [g,x]\mapsto [g,f(x)]\in Y_G$. Since $E$ preserves isomorphisms, then $(X_G,\mu^\theta)$ and $(Y_G,\mu^\eta)$ are $G$-homeomorphic whenever $(X,\theta)$ and $(Y,\eta)$ are $G$-homeomorphic.

We proceed with the next.

\begin{pro}\label{glof} {\rm
   The functor $E:\mathcal{PA}_G\rightarrow \mathcal{A}_G$ is left adjoint to the inclusion functor  $j:\mathcal{A}_G\rightarrow \mathcal{PA}_G,$ that is $\mathcal{A}_G$  is a reflective subcategory of $\mathcal{PA}_G$.   } 
 \end{pro}
 
 \section{Twisted products}\label{tp}
 In this section, we deal with twisted products and use them to obtain structural properties and  relations between objects in  categories related to partial actions.
\subsection{Construction of twisted products}
Let $(X,\theta)$ be a $G$-space and set $G^x=\{g\in G: (g,x)\in  G*X\}$. Consider the  {\it orbit equivalence relation} $\sim_G$ on $X$ given by:
\begin{equation}\label{porbit}x\sim_G y\Longleftrightarrow\,\, \exists g\in G  \,\,{\rm such \,\,that}\,\,g\in G^x\,{\rm and}\, g\cdot x=y,\end{equation}
for each $x,y\in X.$  The orbit space of $X$, denoted by $X/G$, is the set of equivalence classes endowed with the quotient topology. Its elements are the {\it orbits} $[x]=G^x\cdot x$, for each  $x\in X$. It follows by  \cite[Lemma 3.2]{PU1} that the projection map
\begin{equation}\label{qmap}\pi_G: X\ni x\mapsto [x] \in X/G,\end{equation} is continuous and open.

 

Let $G$ be a topological group, $K$ be a subgroup of $G$ and $(X,\theta)$ be a $K$-space. Consider the global action $K\times G\ni (k,g)\mapsto gk\m\in G$ of $K$ on $G$, and let $G\times _K X= (G\times X)/K$ be  the corresponding  orbit space, where $G\times X$ is equipped with the diagonal partial action 
$$K* (G\times X)\ni(k,(g,x))\mapsto (gk\m,\theta(g,x))\in G\times X,$$ (see (\ref{treta})). Therefore 
$G\times _K X=\{[g,x]_K\mid (g,x)\in G\times X\},$ where \begin{equation}\label{clase}
[g,x]_K=\{(gk\m,\theta(k,x)): k\in K^x\},
\end{equation}  The $K$-space $G\times _K X$ is called the {\it twisted product} of the $K$-spaces $G$ and $X.$

We proceed with the following proposition.

\begin{pro}\label{globprod}\rm{\cite[Theorem 3.3]{PU1}
    Let $(X,\theta)$ be a $G$-space. Then $[g,x]=[g,x]_G$, for each $(g,x)\in G\times X$, and consequently $G\times_G X=X_G,$ moreover if 
    
\begin{equation}\label{qo2}
    p^\theta_K: G\times X\ni(g,x)\mapsto [g,x]_K\in G\times_KX, 
\end{equation}     we have that
    $p^\theta=p^\theta_G,$ where $p^\theta$ is defined in (\ref{qo}).}
\end{pro}

The action (\ref{action}) is extended to $G\times_KX$ as follows. 

\begin{pro}\label{tpro}
 \rm{Let $G$ be a topological group, $K$ be a subgroup of $G$ and $(X,\theta)$ be a $K$-space. Then the map:
 \begin{equation}\label{action2}
     \mu^\theta_K: G\times (G\times_KX)\ni(g',[g,x]_K)\mapsto [g'g,x]_K\in G\times_KX, 
 \end{equation}
 is a continuous action of $G$ on $G\times_KX$. Moreover, $\mu^\theta_G=\mu^\theta$, where $\mu^\theta$ is defined in (\ref{action}). }
\end{pro}
\begin{proof}
    First, we shall prove that $\mu^\theta_K$ is well defined. Take $g'\in G$, $[g,x]_K\in G\times_KX$ and $(h,z)\in[g,x]_K$. By (\ref{clase}) there exists $k\in G^x$ such that $(h,z)=(gk\m,\theta(k,x))$. But $[g'gk\m,\theta(k,x)]_K=[g'g,x]_K$ because

    \begin{center}
        $k\m\cdot (g'gk\m,\theta(k,x))=(g'gk\m k,\theta(k\m,\theta(k,x)))\stackrel{\rm{(PA1)}}=(g'g,x)$,
    \end{center}
    then $\mu^\theta_K$ is well defined.  We shall now prove that $\mu^\theta_K$ is continuous. For this,  consider the map $h: G\times (G\times X)\ni(g',(g,x))\mapsto (g'g,x)\in G\times X$, and observe that the following diagram is commutative:
\begin{center}
     $\xymatrix{G\times (G\times X)\ar[d]_-{1_G\times p^\theta_K}\ar[r]^-{h}& G\times X\ar[d]^-{p^\theta_K} \\
    G\times(G\times_K X) \ar[r]_-{\mu^\theta_K}& G\times_KX
    }$
\end{center}
Since $p^\theta_K$ is open, we conclude that $\mu^\theta_K$ is continuous.
\end{proof}
When there is no chance for confusion, we remove $\theta$ in (\ref{qo2}) and (\ref{action2}). 
Furthermore, regarding  Proposition \ref{Kellen1} we have the next result.

\begin{pro}\label{iota2}
 \rm{Let $G$ be a topological group, $K$ be a subgroup of $G$ and $(X,\theta)$ be a $K$-space. Then the map:
 \begin{equation}
     \iota^\theta_K: X\ni x\mapsto [e,x]_K\in G\times_KX,
 \end{equation}
is injective and a $K$-isovariant map, where $G\times_KX$ is considered as a $K$-space by restricting the action $\mu_K$ to $K\times (G\times_KX).$ Moreover, the following assertions hold:
 \begin{enumerate}
     \item [(i)] $\iota^\theta_K(X)$ is open (closed) in $G\times_KX$ if and only if $K*X$ is open (closed) in $G\times X$.
     \item [(ii)] The map $\iota^\theta_K: X\rightarrow \iota^\theta_K(X)$ is a homeomorphism.
     \item [(iii)] If $K=G$, then $\iota^\theta_G=\iota^\theta,$ where $\iota^\theta$ is defined in (\ref{iota}).
 \end{enumerate}}
\end{pro}
\begin{proof}

    Throughout the proof we write $\iota_K$ instead of $\iota^\theta_K$. Consider the map $j_e: X\ni x\mapsto (e,x)\in G\times X$. Then $\iota_K$ is continuous because $\iota_K=p_K\circ j_e$. On the other hand, take $x, y\in X$ such that $\iota_K(x)=\iota_K(y)$. Then  there exists $k\in G^x$ such that $k\cdot(e,x)=(e,y)$, therefore $k=e$ and $y=\theta(e,x)\stackrel{\rm{(PA3)}}=x,$ and $\iota_K$ is injective. Now, to  
 check that $\iota_K$ is $K$-isovariant, observe first that $\iota_K$ is a $K$-map. Indeed, take $(k,x)\in K*X$. Then $$\mu_K(k,\iota_K(x))=\mu_K(k,[e,x]_K)=[k,x]_K=[e,\theta(k,x)]_K=\iota_K(\theta(k,x)),$$ and  $\iota_K$ is a $K$-map. 
    Now take $x\in X$. Since $K_x\subseteq K_{\iota_K(x)}$, it only remains to verify that $K_{\iota_K(x)}\subseteq K_x$. Let $k\in K_{\iota_K(x)}$. Then $\mu_K(k,\iota_K(x))=\iota_K(x)$ and there exists $h\in K^x$ such that $h\cdot (k,x)=(e,x)$. Thus $(kh\m,\theta(h,x))=(e,x)$ and $h=k$, then $k\in K^x$ and $\theta(k,x)=\theta(h,x)=x$.  This shows that $k\in K_x$ and $\iota_K$ is $K$-isovariant.

    Before passing to the proof  of (i), (ii), and (iii) observe that
    \begin{equation}\label{preimagen}
        p_K\m(\iota_K(X))=K*X
    \end{equation}
    Indeed, take $(k,x)\in K*X$. Then
    \begin{center}
        $p_K(k,x)=[k,x]_K=[e,\theta(k,x)]_K\in \iota_K(X)$,
    \end{center}
    therefore $K*X\subseteq p_K\m(\iota_K(X))$. On the other hand, let $(g,x)\in p_K\m(\iota_K(X))$. There exists $z\in X$ such that $[g,x]_K=[e,z]_K$ and there is $k\in K^x$ with $k\cdot(g,x)=(e,z)$, whence $g=k$ and $(g,x)\in K*X$.
    
    $(i)$ Since $p_K$ is a quotient map, then $\iota_K(X)$ is open (closed) in $G\times_KX$ if and only if $K*X$ is open (closed) in $G\times X$ by (\ref{preimagen}).

    $(ii)$ Notice that $p_K: p_K\m(\iota_K(X))\rightarrow \iota_K(X)$ is a quotient map because $p_K$ is an open map, (see \cite[Chapter VI, Theorem 2.1]{JD}). But $p_K\m(\iota_K(X))=K*X$ then $p_K: K*X\rightarrow \iota_K(X)$, $(k,x)\mapsto [k,x]_K=[e,\theta(k,x)]_K$, is a quotient map. On the other hand, observe that $\theta: K*X\rightarrow X$ is constant in the fibers of $p_K$. Indeed, take $x\in X$ and $z\in p_K\m(\iota_K(x))$. By (\ref{clase}) there is $k\in K^x$ such that $z=(k\m,\theta(k,x))$ and we have $\theta(z)=\theta(k\m,\theta(k,x))\stackrel{\rm{(PA1)}}=x$. Then $\theta$ is constant in the fiber $p_K\m(\iota_K(x))$. By the Transgression Theorem \cite[Chapter VI, Theorem 3.2]{JD} there exists a continuous map $j: \iota_K(X)\rightarrow X$ such that the following diagram is commutative:

\begin{center}
     $\xymatrix{K*X\ar[d]_-{p_K}\ar[r]^-{\theta}& X\\
    \iota_K(X) \ar[ur]_-{j}& 
    }$
\end{center}
then $j(\iota_K(x))=j([e,x]_K)=j(p_K(e,x))=\theta(e,x)=x$, for each $x\in X$. We conclude that $j\circ \iota_X=\rm{id}_X$. Further, given $z\in \iota_K(X)$ there exists $x\in X$ such that $z=[e,x]_K$  and we have 
\begin{center}
    $\iota_K(j(z))=\iota_K(j([e,x]_K))=\iota_K(j(p_K(e,x)))=\iota_K(\theta(e,x))\stackrel{\rm{(PA3)}}=\iota_K(x)=[e,x]_K=z$,
\end{center} 
then $\iota_K\circ j=\rm{id}_{\iota_K(X)}$ and $j$ is the inverse of $\iota_K$, thus $\iota_K:X\rightarrow \iota_K(X)$ is a homeomorphism, as desired.
\end{proof}


Now we state the last result of this section.
\begin{pro}\label{iota3}{\rm
	Let $G$ be a topological group, $K$ be a subgroup of $G$, and let $(X,\theta)$ be a $K$-space. Then $G\times_K(X_K)$ and $G\times_KX$ are $G$-homeomorphic.}
 \end{pro}
\begin{proof}
	We know that $X_K=K\times_K X,$ because of   Proposition \ref{globprod}. Consider the map $m: G\times_K(K\times X)\ni ([g,(h,x)]_K)\mapsto [gh,x]_K\in G\times_KX,$ and the map $n: G\times_KX\ni [g,x]_K\mapsto[g,(e,x)]_K\in G\times_K(K\times_KX).$  Notice that $m$ and $n$ are well defined  $G$-maps because the following diagrams are commutative:
\begin{multicols}{2}
	\begin{center}
		$\xymatrix{G\times (K\times X)\ar[d]_-{1_G\times p^\theta_K}\ar[r]^-{j}& G\times X\ar[dd]^-{p^\theta_K} \\
			G\times(K\times_K X) \ar[d]_-{p}& \\ G\times_K(K\times_KX) \ar[r]^-{m} & G\times_KX
		}$
	\end{center}
	\begin{center}
		$\xymatrix{ G\times X \ar[dd]_-{p^\theta_K} \ar[r]^-{j'} & G\times (K\times X)\ar[d]^-{1_G\times p^\theta_K}  \\
			& G\times(K\times_K X) \ar[d]^-{p} \\ 
			G\times_KX \ar[r]^-{n} & G\times_K(K\times_KX)  
		}$
	\end{center}
\end{multicols}	
 where $p$ is the orbital projection, $j(g,(h,x))=(gh,x)$ and $j'(g,x)=(g,(e,x))$, for each $g, h\in G$ and $x\in X$. It is clear that $m$ is the inverse of $n$, then we conclude that $G\times_K(X_K)$ and $G\times_KX$ are $G$-homeomorphic.
 \end{proof}
\subsection{Adjoint functors}

Let $G$ be a topological group and let $K$ be a subgroup of $G$. The construction of twisted products induces a functor $G\times_K-:\mathcal{PA}_K\rightarrow \mathcal{A}_G$ which coincides with the globalization functor $E: \mathcal{PA}_G\rightarrow \mathcal{A}_G$ when $G=K$, as follows:

\noindent Given an object $(X,\theta)$ in $\mathcal{PA}_K$, we define  $(G\times_K-)(X,\theta)=(G\times_KX,\mu_K)$, where $\mu_K$ is defined in Proposition \ref{action2}.  Now, for  objects $(X,\theta)$ and $(Y,\eta)$ in $\mathcal{PA}_K$, and a $K$-map $f:X\rightarrow Y$, the functor $G\times_K-$ sends $f$ to $(G\times_K-)(f):=G\times_Kf,$ where 
$$G\times_Kf: G\times_KX\ni [g,x]_K\mapsto [g,f(x)]_K\in G\times_KY.$$
We notice that $G\times_Kf$ is well defined. Indeed, take $(g,x)\in G\times X$ and $(h,y)\in [g,x]_K$. By (\ref{clase}) there exists $k\in K^x$ such that $(h,y)=(gk\m,\theta(k,x))$. But $f$ is a $K$-map, then

\begin{center}
    $[gk\m,f(\theta(k,x))]_K=[gk\m,\eta(k,f(x))]_K\stackrel{\rm{(PA1)}}=[gk\m k,\eta(k\m,\eta(k,f(x)))]_K\stackrel{\rm{(PA1)}}=[g,f(x)]_K$,
\end{center}
and $G\times_Kf$ is well defined.  Now consider the following commutative diagram:

\begin{center}
     $\xymatrix{ G\times X\ar[d]_-{p^\theta_K}\ar[r]^-{1_G \times f}& G\times Y\ar[d]^-{p^\eta_K} \\
    G\times_K X \ar[r]_-{G\times_Kf}& G\times_KY
    }$
\end{center}
Since $p^\theta_K$ and $p^\eta_K$ are open maps, then $G\times_Kf$ is continuous. Moreover, restricting the group action from $G$ to $K$ giveses a functor 
\begin{center}
  $\rm{res}$$ ^G_K: \mathcal{A}_G\rightarrow \mathcal{PA}_K$,
\end{center}
which coincides with the inclusion functor when $G=K$. We have the next result.


\begin{teo}\label{glof2}
{\rm    Let $G$ be a topological group and $K$ be a subgroup of $G$. Then $G\times_K-$ is left adjoint to $\rm{res}$$^G_K$.}
\end{teo}
\begin{proof}
Let $(X,\theta)$ be an object in $\mathcal{PA}_K$ and $(Y,\eta)$ be an object in $\mathcal{A}_G$. We define the following map:
$$\lambda_{(X,Y)}:\mathcal{A}_G(G\times_KX,Y)\ni f\mapsto f\circ\iota_K\in \mathcal{PA}_K(X,\ {\rm res} ^G_K(Y)),$$
where $\iota_K$ is defined in (\ref{iota2}). Since $f$ and $\iota_K$ are $K$-maps, then $\lambda_{(X,Y)}(f)$ is a $K$-map. Now consider the map

\begin{center}
    $\tau_{(X,Y)}:\mathcal{PA}_K(X,\ $\rm{res}$ ^G_K(Y))\ni f\mapsto \tau_{(X,Y)}(f) \in \mathcal{A}_G(G\times_KX,Y)$, 
\end{center}
where $\tau_{(X,Y)}(f)([g,x]_K)=\eta(g,f(x)),$ for each $[g,x]_K\in G\times_KX.$  We check  that $\tau_{(X,Y)}(f)$ is well defined. Indeed, take $(g,x)\in G\times X$ and $(h,y)\in [g,x]_K$, then there exists $k\in K^x$ such that $(h,y)=(gk\m,\theta(k,x))$ and
\begin{center}
 $\tau_{(X,Y)}(f)([h,y]_K)=\tau_{(X,Y)}(f)([gk\m,\theta(k,x)]_K)=\eta(gk\m,f(\theta(k,x)))=\eta(g,f(x)),$  
\end{center}
as desired. Further, consider the following commutative diagram:

\begin{center}
     $\xymatrix{G\times X\ar[d]_-{p^\theta_K}\ar[r]^-{m}& Y\\
    G\times_KX \ar[ur]_-{\tau_{(X,Y)}(f)}& 
    }$
\end{center}
where $m(g,x)=\eta(g,f(x))$, for each $(g,x)\in G\times X$. Since $p^\theta_K$ is open, then $\tau_{(X,Y)}(f)$ is continuous. Finally, it is clear that $\tau_{(X,Y)}(f)$ is a $G$-map. Now we shall check that $\lambda_{(X,Y)}$ is the inverse of $\tau_{(X,Y)},$
let $f\in \mathcal{A}_G(G\times_KX,Y)$, and $(g,x)\in G\times _KX$. Then:
\begin{center}
    $\tau_{(X,Y)}(\lambda_{(X,Y)}(f))([g,x]_K)=\eta(g,\lambda_{(X,Y)}(f)(x))=\eta(g,f([e,x]_K))=f([g,x]_K)$
\end{center}
and $\tau_{(X,Y)}(\lambda_{(X,Y)}(f))=f$. Therefore $\lambda_{(X,Y)}$ is a right inverse of $\tau_{(X,Y)}$. Now if $f\in \mathcal{PA}_K(X,\ $\rm{res}$ ^G_K(Y))$ and  $x\in X$, we have 
\begin{equation*}
    \begin{aligned}
        \lambda_{(X,Y)}(\tau_{(X,Y)}(f))(x) & =\tau_{(X,Y)}(f)([e,x]_K)
        & =\eta(e,f(x)))
        & =f(x),
    \end{aligned}
\end{equation*}
then $\lambda_{(X,Y)}(\tau_{(X,Y)}(f))=f$ and $\lambda_{(X,Y)}$ is a left inverse of $\tau_{(X,Y)}$. We conclude that $\lambda_{(X,Y)}$ is a bijective map.

Take $(X',\theta')\in (\mathcal{PA}_K)_0,$ $(Y',\eta')\in (\mathcal{A}_G)_0$, $r\in \mathcal{PA}_K(X',X)$ and $s\in \mathcal{A}_G(Y,Y')$. Then the following diagram is commutative:

\begin{center}
    $\xymatrix{ \mathcal{A}_G(G\times_KX,Y)\ar[d]_-{s_*}\ar[r]^-{\lambda_{(X,Y)}}& \mathcal{PA}_K(X,\rm{res}^\textit{G}_\textit{K}(Y))\ar[d]^-{\rm{res}^\textit{G}_\textit{K}(s)_*} \\
    \mathcal{A}_G(G\times_KX,Y') \ar[r]_-{\lambda_{(X,Y')}}& \mathcal{PA}_K(X,\rm{res}^\textit{G}_\textit{K}(Y'))
    }$
\end{center}
Indeed, for each $f\in \mathcal{A}_G(G\times_KX,Y)$ and $x\in X$ we have
\begin{equation*}
    \begin{aligned}
        \textrm{res}^\textit{G}_\textit{K}(s)_* (\lambda_{(X,Y)}(f))(x)& =\textrm{res}^\textit{G}_\textit{K}(s)(\lambda_{(X,Y)}(f))(x)\\
        & =\textrm{res}^\textit{G}_\textit{K}(s)(f([e,x]_K))\\
        & =(s\circ f)([e,x]_K))\\
        & =\lambda_{(X,Y')}(s\circ f)(x)\\
        & = \lambda_{(X,Y')}(s_*(f))(x).
    \end{aligned}
\end{equation*}

\noindent Similarly, the following diagram is commutative:
\begin{center}
    $\xymatrix{ \mathcal{A}_G(G\times_KX,Y)\ar[d]_-{(G\times_Kr)^*}\ar[r]^-{\lambda_{(X,Y)}}& \mathcal{PA}_K(X,\rm{res}^\textit{G}_\textit{K}(Y))\ar[d]^-{r^*} \\
    \mathcal{A}_G(G\times_KX',Y) \ar[r]_-{\lambda_{(X',Y)}}& \mathcal{PA}_K(X',\rm{res}^\textit{G}_\textit{K}(Y))
    }$
\end{center}
Indeed, take $l\in \mathcal{A}_G(G\times_KX,Y)$ and $x\in X'$. Then

\begin{equation*}
    \begin{aligned}
        r^*(\lambda_{(X,Y)}(f))(x)& =(\lambda_{(X,Y)}(f)\circ r)(x)\\
        & =f([e,r(x)]_K)\\
        & =(f\circ G\times_Kr)([e,x]_K)\\
        & =\lambda_{(X',Y)}(f\circ G\times_Kr)(x)\\
        & = \lambda_{(X',Y)}((G\times_Kr)^*(f))(x),
    \end{aligned}
\end{equation*}
this completes the proof.
\end{proof}
\begin{coro}\label{equivalencia}
    Let $G$ be a topological group, $K$ be a subgroup of $G$ and $\{(X_j,\theta^j): j\in J\}$ be a family of $K$-spaces, with $J$ finite. Put $X=\prod\limits_{j\in J}X_j$ with the diagonal partial action of $K$. Then the following map is a $G$-homeomorphism:
    \begin{center}
        $G\times_K X\ni [g,(x_1,x_2,\dots,x_n)]_K\mapsto ([g,x_1]_K,[g,x_2]_K,\dots,[g,x_n]_K)\in \prod\limits_{j\in J}(G\times_KX_j)$,
    \end{center}
    In particular, if $K=G$ then
    \begin{equation}\label{pdia}
        X_G\ni[g,(x_1,x_2,\dots,x_n)]\mapsto ([g,x_1],[g,x_2],\dots,[g,x_n])\in \prod\limits_{j\in J}(X_j)_G, \end{equation}
    is a $G$-homeomorphism.
\end{coro}

\subsection{Metrizability of twisted products}

In this section, sufficient conditions are established for the metrizability of twisted products. We start with the next result.

\begin{lem}\label{T1}
    Let $G$ be a topological group, $K$ be a subgroup of $G$ and $(X,\theta)$ be a  $K$-space.  Suppose that  $X$ is Hausdorff, then  the following statements hold:  \begin{enumerate}
        \item [(i)] Let $x\in X,$ if $K^x$ is closed in $G$, then $(p^\theta_K)^{-1}(p_K^\theta(g,x))$ is closed in $G\times X$, for each $g\in G,$ in particular the space $G\times_KX$ is $T_1$ provided that $K^x$ is closed in $G$, for each $x\in X$.       
    \end{enumerate}
    Now suppose that $X$ and $G$ are  second countable, then:
\begin{enumerate}
        \item [(ii)] $G\times_KX$ is metrizable if and only if $G\times_KX$ is regular and $T_1$.            
    \end{enumerate}

        \end{lem}
\begin{proof}
 (i) Take $(h,y)\in \overline {(p^\theta_K)^{-1}(p_K^\theta(g,x)) }$. By (\ref{clase}) there exists a net  $\{k_\lambda\}_{\lambda\in \Lambda}$ in $K^x$ such that $\lim (gk\m_\lambda, \theta(k_\lambda,x))= (h,y)$. Then $\lim k_\lambda= h\m g\in K^x$ and $y=\lim \theta(k_\lambda,x)=\theta(h\m g,x),$ from this we get that $$(h,y)=(g(h\m g)\m,\theta(h\m g, x))\in (p^\theta_K)^{-1}(p_K^\theta(g,x)),$$  thus  $(p^\theta_K)^{-1}(p_K^\theta(g,x))$ is closed in $G\times X$.


 \noindent(ii) Since $G\times X$ is second countable and $p_K^\theta:G\times X\rightarrow G\times_KX$ is open, then $G\times_KX$ second countable. By Urysohn’s metrization Theorem, the space $G\times_KX$ is metrizable whenever it is regular and $T_1$.
    \end{proof}
We recall the next result.
\begin{lem}\label{regular}\rm{\cite[Theorem 4.6]{PU1}
    Let $G$ be a topological group, $K$ be a subgroup of $G$ and $(X,\theta)$ be a regular $K$-space such that $\theta$ is nice. Then $G\times_KX$ is regular under any of the following conditions:
    \begin{enumerate}
        \item [(i)] $X_k$ is closed in $X$,  for each $k\in K$.
        \item [(ii)] $X$ is a locally compact metric space and $R$ is closed in $(K\times X)^2$, where $R$ is defined in (\ref{eqgl}).
    \end{enumerate}  }
\end{lem}
We finish this section with the following result.
\begin{teo}\label{mett}
    \rm{Let $G$ be a separable metrizable group, $K$ be a subgroup of $G$ and $(X,\theta)$ be a separable metrizable $K$-space such that $\theta$ is nice. Then $G\times_KX$ is metrizable under any of the following conditions:
    \begin{enumerate}
        \item [(i)] $K*X$ is closed in $G\times X$.
        \item [(ii)] $G\times_KX$ is $T_1$, $X$ is locally compact and $R$ is closed in $(K\times X)^2$.

    \end{enumerate} }
\end{teo}
\begin{proof}
    (i) Let $x\in X$. Since $K*X$ is closed in $G\times X$ then $K^x$ is closed in $G$ and $X_k$ is closed in $X$. Thus the space $G\times_KX$ is regular and $T_1$ by (i) of  Lemma \ref{T1} and Lemma \ref{regular}. Then $G\times_KX$ is metrizable by (ii) of Lemma \ref{T1}.

  \noindent  (ii) Since the space $G\times_KX$ is regular by Lemma \ref{regular},  the result follows by (ii) of Lemma \ref{T1}.
\end{proof}
\section{Homotopy properties of twisted products}\label{gc}
Now we recall some notions that will be useful in this chapter, for more details the interested reader may consult \cite{VM}. 
Let $X$ be a topological space and $A$ be a subspace of $X,$  a continuous map $ r\colon X\to A$ is a {\it retraction} if the restriction of $r$ to $A$ is the identity map on $A.$ In that case, $A$ is called a {\it retract} of $X$.

Now, for $n\geq 0$  an integer, the space $X$ is called {\it connected in dimension $n$} or ${\rm C}^n$, provided that for any integer $m$ with  $0\leq m\leq n$, every continuous function $f:S^m\rightarrow X$ extends to a continuous function $F: B^{m+1}\rightarrow X,$ where $S^m$ denotes the unit $m$-sphere and $B^{m+1}$ the unit $(m+1)$-ball. If $X$ is ${\rm C}^n$ for each $n\geq 0$, then $X$ is called {\it homotopically trivial}. In addition, $X$ is said to be {\it locally connected in dimension $n$} (or ${\rm LC}^n$), provided that for every $x\in X$ and for every $U$ neighborhood of $x$ in $X$ and for every $0\leq m \leq n$, there exists a neighborhood $V$ of $x$ in $X$ so that $V\subseteq U$ and every continuous function $f:S^m\rightarrow V$ extends to a continuous function $F:B^{m+1}\rightarrow U$. Furthermore, we recall that $X$ is {\it contractible} provided that the identity map on $X$ is null-homotopic, i.e. if it is homotopic to some constant map. In addition, $X$ is called {\it locally contractible} provided that for each $x\in X$ and $U$ neighborhood of $x$ in $X$, there exists a neighborhood $V$ of $x$ in $X$ so that $V\subseteq U$ and $V$ is contractible in $U$, that is, there exists a continuous map $H:V\times I\rightarrow U$, where $H_0$ is the inclusion map, $H_1$ a constant map, where $I=[0,1]$ and $H_t:X\ni x\mapsto H(x,t)\in U,$ for any $t\in I$.

We start this section with an example where $X$ is not a retract of $X_G$ and $X_G$ is neither contractible nor homotopically trivial, nor $\rm{C}^1$ even though $X$ is.

\begin{ejem}\label{ex2}\cite[P. 175]{CL}\label{matriz}
	Consider $G=\rm{GL}(2;\mathbb{R})$ acting  partially   on $\mathbb{R}$ by setting  $$G*\mathbb R=\left\{(g,x)\in G\times \mathbb R\mid g=\begin{pmatrix}
			a & b \\
			c & d 
		\end{pmatrix} \text{and } cx+d\neq 0\right\}$$ and 
	$$\theta:G*\mathbb R\ni (g,x)\mapsto \frac{ax+b}{cx+d}\in \mathbb R$$		
	Then  $\mathbb{R}_G$ is homeomorphic to $\mathbb{R}\cup\{\infty\}\simeq  S^1$ thus it is not contractible. Moreover, $\mathbb{R}_G$ is neither homotopically trivial nor $\rm{C}^1$.
\end{ejem} 

Regarding Example \ref{ex2}, we have the following proposition.
\begin{pro}\label{lcn}{\rm 
    Let $(X,\theta)$ be a $G$-space, $X_G$ be its enveloping space and $n\geq 0$ be an integer. If $\theta$ is nice, then the following statements hold:
    \begin{enumerate}
        \item [(a)] $X$ is locally contractible, if and only if, $X_G$ is locally contractible.
        \item [(b)] $X$ is ${\rm LC}^n,$ if and only if, $X_G$ is ${\rm LC}^n.$
    \end{enumerate}}
\end{pro}
\begin{proof} Let $z\in X_G$ and let  $U$ be a neighborhood of $z$. Take $(g, z)\in G\times X$ so that $z=[g,x]$ and consider the map $\iota$ defined in \eqref{iota}. Then $\mu_{g^{-1}}(U)$ is a neighborhood of $\iota(x)$ in $X_G$ and since $\iota$ is open, the set  $V=\iota(X)\cap \mu_{g^{-1}}(U)$ is an open neighborhood of $\iota(x)$ in $\iota(X).$ We shall use these sets in the proof of both items.\\    
 For (a), suppose that $X$ is locally contractible. Let $W=\iota^{-1}(V)$. Since $X$ is locally contractible, there exists $O$ an open neighborhood of $x$ in $X$ such that $O\subseteq W$ and $O$ is contractible in $W$. Let $H: O\times I\rightarrow W$ be a homotopy such that $H_0$ is the inclusion map and $H_1(o)=w_x \in W,$ for each $o\in O.$ 
Set    $$F:\mu_g(\iota(O))\times I\ni (p,t)\mapsto (\mu_g\circ \iota\circ H)(\iota^{-1}\mu_{g^{-1}}(p),t) \in \mu_g(\iota(W)).$$ 
    Then $F$ is continuous because is a composition of continuous maps. Observe that $\mu_g(\iota(W))\subseteq \mu_g(V)\subseteq U$, $F_0$ is the inclusion and $F_1(p)=[g,w_x]$ for each $p\in \mu_g(\iota(O))$. Then by composing $F$ with the inclusion of $\mu_g(\iota(W))\hookrightarrow U$ we obtain the desired function. 
    
    Conversely, suppose that $X_G$ is locally contractible. Since $\iota: X\to X_G$ is open and a homeomorphism onto $\iota(X)$ we have that $\iota(X)$ and thus $X$ is locally contractible.

For item (b), consider $0\leq m\leq n$. Then there exists an open  neighborhood $W$ of $\iota(x)$ in $\iota(X)$ so that $W\subseteq V$ and each continuous map $f:S^m\rightarrow W$ extends to a continuous map $F: B^{m+1}\rightarrow V$.  Observe that $z=[g,x]\in \mu_g(W)\subseteq \mu_g(V)\subseteq U$. Let $f: S^m\rightarrow \mu_g(W)$ be a continuous map. Since $\mu_{g^{-1}}\circ f: S^m\rightarrow W$ is continuous, there exists $\tilde{f}: B^{m+1}\rightarrow V$ a continuous extension of $\mu_{g^{-1}}\circ f$. Let $F=\mu_g\circ \tilde{f}: B^{m+1}\rightarrow U$ and  $a\in S^m$. Then $F(a)=\mu_g(\tilde{f}(a))=\mu_g(\mu_{g^{-1}}(f(a)))=f(a)$, and $F$ extends  $f$, from this we conclude that $X_G$ is $\rm{LC}^n$.
    \end{proof}

 Let $(X,\theta)$ be $G$-space and $K$ be a subgroup of $G$. The set of $K$-fixed points is 
 \begin{equation}\label{xh}
 	X[K]=\{x\in X: K\subseteq G_x\}.
 \end{equation} 
An element $w\in X[G]$ is called a {\it fixed point}. Further, if $(Y,\eta)$ is another $G$-space. Two continuous $G$-maps $m,n: X\rightarrow Y$ are {\it $G$-homotopic} provided that there is a $G$-map $H:X\times I\rightarrow Y$, such that $H_0=m$ and $H_1=n$, where $H_t: X\ni x\mapsto H(x,t)\in Y$, for each $t\in I$. We always consider $X\times I$ as a $G$-space with the diagonal partial action, where $G$ acts on $I$ trivially. The $G$-map $H$ is called a {\it $G$-homotopy} from $m$ to $n$. We say that   $X$ is {\it  $G$-contractible} if there exists a $G$-homotopy $H:X\times I\rightarrow X$ with $H_1={\rm id}_X$ and $H_0(x)=w,$ for each $x\in X$ and some  fixed point  $w\in X.$ 
\begin{ejem}\label{contra}
    Consider $G=\mathbb{R}^*$ as a multiplicative group. Then $G$ acts globally on  $\mathbb{R}^n$ via $\mu: G\times \mathbb{R}^n\ni (r,x)\mapsto rx\in \mathbb{R}^n.$ Let $X=\{x\in \mathbb{R}^n: ||x||<1\}$ and let $\theta$ be the restriction of $\mu$ to $X,$ 
    then follows by Proposition \ref{isog}  that $X_G$ is topologically equivalent to $\mathbb{R}^n,$  because $\mathbb{R}^n=\bigcup\limits_{r\in \mathbb{R}^*}\mu_r(X),$ in particular $X$ and  $X_G$ are contractible. Further,  we affirm that the map $H:X\times I\ni (x,t)\mapsto tx\in X$, 
    is a $G$-contraction to the fixed point $0$. Take $t\in I$ and $(r,x)\in G*X$, then $||rx||<1$ and $||rH_t(x)||=||t(rx)||=t||rx||<1$. Therefore $(r,H_t(x))\in G*X$ and $H_t(\theta_r(x))=trx=rtx=\theta_r(H_t(x))$  which implies that  $X$ is $G$-contractible. Observe that $X_G=\mathbb R^n$ is also $G$-contractible.
\end{ejem}
Regarding Example \ref{ex2} and Example \ref{contra}, it is natural to wonder if $X_G$ is $G$-contractible when $X$ is $G$-contractible. To answer this question we need first a couple of results.
\begin{lem}\label{homeocanoni}		Let $G$ be a topological group, $K$ be a subgroup of $G$ and $(Y,\theta)$ be a $K$-space, where $\theta$ is a trivial partial action, that is $\theta(k,x)=x$ for each $(k,x)\in K* X$. Then $G\times_KY$ and $Y$ are homeomorphic.
\end{lem}
\begin{proof}
	Consider the map $\delta: G\times_KY\ni [g,y]_K\mapsto y\in  Y$. Notice that $\delta$ is well-defined and continuous because the following diagram is commutative:
	\begin{center}
		$\xymatrix{G\times Y\ar[d]_-{p_K}\ar[r]^-{\text{proj}_2}& Y\\
			G\times_KY \ar[ur]_-{\delta}& 
		}$
	\end{center}
Moreover the inverse of $\delta$ is given by  $\delta\m: Y\ni y\mapsto[e,y]_K\in G\times_KY$, then $\delta$ is a homeomorphism.
\end{proof}
\begin{pro}\label{Ghomot}{\rm
	Let $G$ be a topological group, $K$ be a subgroup of $G$, $(X,\theta)$ and $(Y,\eta)$ be $K$-spaces. If two $K$-maps $m,n: X\rightarrow Y$ are $K$-homotopic, then $G\times_Km$ and $G\times_Kn$ are $G$-homotopic. In other words, the functor $G\times_K-:\mathcal{PA}_K\rightarrow \mathcal{A}_G$ preserves homotopies. In particular, the globalization functor $E:\mathcal{PA}_G\rightarrow \mathcal{A}_G$ preserves homotopies.}
\end{pro}
\begin{proof}
	Let $F:X\times I\rightarrow Y$ be a $K$-homotopy from $m$ to $n.$ Consider the $G$-map $G\times_KF: G\times_K(X\times I)\rightarrow G\times_KY$ induced by the functor $G\times_K-$. By Corollary \ref{equivalencia} the map $\alpha: G\times_K(X\times I) \ni [g,(x,t)]_K\mapsto ([g,x]_K,[g,t]_K)\in (G\times_KX)\times(G\times_K I) $  is a $G$-homeomorphism. Moreover, since we are considering $I$  endowed with a trivial partial action, it follows from Lemma \ref{homeocanoni}  that the map $\beta: I\ni t\mapsto[e,t]_K\in G\times_KI$  is a $G$-homeomorphism. Then we can consider the $G$-map $$\tilde{F}:=(G\times_KF)\circ \alpha\m\circ (\text{id}\times\beta):(G\times_KX)\times I\rightarrow G\times_KY.$$
	
	Let us show that $\tilde{F}$ is a $G$-homotopy between $G\times_Km$ and $G\times_Kn.$ Indeed, take $[g,x]_K\in G\times_KX$. Then:
	
	\begin{equation*}
		\begin{aligned}
			\tilde{F}([g,x]_K,0) & =(G\times_KF)\circ \alpha\m([g,x]_K,[e,0]_K)\\
			& =(G\times_KF)\circ \alpha\m([g,x]_K,[g,0]_K)\\
			& =(G\times_KF)([g,(x,0)]_K)\\
			& =[g,F(x,0)]_K\\
			& = [g,m(x)]_K\\
			& = (G\times_Km)([g,x]_K),
		\end{aligned}
	\end{equation*}
then $\tilde{F}_0=G\times_Km$. In a similar way $\tilde{F}_1=G\times_Kn$. This completes de proof.
\end{proof}
\begin{teo}\label{G-con}{\rm
    Let $X$ be a topological space with a partial action $\theta$ of a topological group $G$. If $X$ is $G$-contractible, then $X_G$ is $G$-contractible.}
\end{teo}
\begin{proof}
	Let $w\in X$ be a fixed point and $F:X\times I\rightarrow X$ be a $G$-homotopy such that $F_0=\text{id}_X$ and $F_1(x)=w$, for each $x\in X$. By Theorem \ref{Ghomot} the map $\tilde{F}: X_G\times I\rightarrow X_G$ is a $G$-homotopy from $G\times_G \text{id}_X$ to $G\times_G F_1$. But $(G\times_G F_1)([g,x])=[g,F_1(x)]=[g,w]=[e,\theta(g,w)]=[e,w]$. Then $G\times_G F_1$ is constant to fixed point $[e,w]$. We conclude that $X_G$ is $G$-contractible.
\end{proof}
\begin{ejem}
    Let $G=\rm{GL}(2;\mathbb{R})$ with the partial action on $\mathbb{R}$ given in Example \ref{matriz}. Observe that $\mathbb{R}_G$ is not $G$-contractible because it is not even contractible. This can also be deduced from the fact that $\mathbb{R}_G$ does not have  $G$-fixed points. For this reason, we conclude that the hypothesis of $G$-contractibility in $X$ cannot be deleted in Theorem \ref{G-con}.
\end{ejem}

We now turn our attention to the local $G$-contractibility. Let $(X,\theta)$ be a $G$-space and $K$ be a subgroup of $G$. A nonempty subset $S$ of $X$ is said to be {\it $K$-invariant} if $\theta(k,s)\in S$ for each $k\in K$ and $s\in S$ such that $(k,s)\in G*X$. If $S$ is $G$-invariant, then we say that $S$ is invariant.

A $G$-space $(X,\theta)$ is said to be {\it locally $G$-contractible} if for each $x\in X$ and $U$ a   $G_x$-invariant  neighborhood of $x$,  there exists a   $G_x$-invariant neighborhood of $V$ of $x$ such that $V\subset U$ and $V$ is $G_x$-contractible in $U$. 

 For the reader's convenience, we recall the next.

\begin{lem}\label{invariant}{\rm
    Let $X$ be a space with a global action of a compact group $G$. If $V$ is a neighborhood of an invariant subset $A\subseteq X$, then an invariant neighborhood $V$ of $A$ exists so that $V\subseteq U$.}
\end{lem}
\begin{teo}\label{locallyc}{\rm
    Let $G$ be a compact group, $(X,\theta)$ be a $G$-space, and let $X_G$ be its enveloping space. If $\theta$ is nice, then $X$ is locally $G$-contractible if and only if $X_G$ locally $G$-contractible.}
\end{teo}
\begin{proof} $\Rightarrow) $ Let $z\in X_G$ and $U$ be a $G_z$-invariant neighborhood of $z$. Take $g\in G$ and $x\in X$ so that $z=[g,x]$. Then $G_z=G_{\mu_g(\iota(x))}=gG_{\iota(x)}g^{-1}$ and $G_{\iota(x)}\mu_{g^{-1}}(U)\subseteq \mu_{g^{-1}}(U)$, therefore $\mu_{g^{-1}}(U)$ is a $G_{\iota(x)}$-invariant neighborhood of $\iota(x)$ in $X_G$. Further, the set $V=\iota(X)\cap \mu_{g^{-1}}(U)$ is an open neighborhood of $\iota(x)$ in $\iota(X)$ and $W=\iota^{-1}(V)$ is a $G_x$-invariant neighborhood of $x$ in $X$. Indeed, to check this last statement take  $h\in G_x$ and $w\in W$ so that $(h,w)\in G*X$. Then
\begin{center}
   $\iota(\theta_h(w))=[1,\theta_h(w)]=[h,w]=\mu_h(\iota(w))\in \mu_h(\mu_{g^{-1}}(U))\subseteq \mu_{g^{-1}}(U)$,
    \end{center}
    therefore we deduce that $\iota(\theta_h(w))\in \iota(X)\cap\mu_{g^{-1}}(U)=V$ and $\theta_h(w)\in W$, therefore $W$ is $G_x$-invariant.    Since $X$ is locally $G$-contractible, there exists a $G_x$-invariant neighborhood $O$  of $x$ in $X$ such that $O\subseteq W$ and $O$ is $G_x$-contractible in $W$. Let $H: O\times I\rightarrow W$ be a $G_x$-homotopy, where $H_0$ be the inclusion map and $H_1:O\ni o\mapsto w_x\in W$, for some $w_x\in W$ a $G_x$-fixed point. 
 Observe that $z\in \mu_g(\iota(O))\subseteq \mu_g(\iota(W))\subseteq U$. Since $G$ is compact and $\{z\}$ is $G_z$-invariant, then  Lemma \ref{invariant} implies that there exist a $G_z$-invariant neighborhood $P$ of $z$ in $X_G$ for which $z\in P\subseteq \mu_g(\iota(O))$.
    Consider the continuous map $$F:P\times I\ni(p,t)\mapsto (\mu_g\circ \iota\circ H)(\iota^{-1}\mu_{g^{-1}}(p),t)\in U,$$
 which is  a $G_z$-homotopy. Indeed, take $t\in I$, we shall check that $F_t:P\rightarrow U$ is a $G_z$-map. Given  $p\in P$ and $r\in G_z$, we need to verify that $\mu_r(F_t(p))=F_t(\mu_r(p))$. Take $o\in O$ and $n\in G_{\iota(x)}$ so that $p=\mu_g(\iota(o))=[g,o]$ and $r=gng^{-1}$. Observe that $F_t(p)=[g,H_t(o)]$ and $\mu_r(F_t(p))=[gn,H_t(o)]$. Since $P$ is $G_z$-invariant, there exists $s\in O$ so that $\mu_r(p)=\mu_g(\iota(s))=[g,s]$, then $F_t(\mu_r(p))=[g,H_t(s)]$. But \begin{center}
        $[g,s]=\mu_r(p)=\mu_{gng^{-1}}([g,o])=[gn,o]$,
    \end{center}
    then by \eqref{eqgl} we get that  $o\in X_{n^{-1}}$ and $\theta_n(o)=s$. Since $H$ is a $G_x$-homotopy we have $H_t(o)\in X_{n^{-1}}$ and $\theta_n(H_t(o))=H_t(\theta_n(o))=H_t(s)$, this shows that:
    \begin{center}
        $\mu_r(F_t(p))=[gn,H_t(o)]=[g,H_t(s)]=F_t(\mu_r(p))$,
    \end{center}
    then $F_t$ is $G_z$-map.  Finally, observe that $F_0$ is the inclusion map and $F_1(p)=[g,w_x]$ for each $p\in P$. This ends the proof.

    $\Leftarrow)$ Let $x\in X$ and let $U$ be a $G_x$-invariant neighborhood of $x$ in $X$. Since  $\iota(U)$ is a neighborhood of $\iota(x)$ and $\{\iota(x)\}$ is $G_{\iota(x)}$-invariant, by Lemma \ref{invariant} and the hypothesis there are $G_{\iota(x)}$-invariant neighborhoods $V$ and $W$  of $\iota(x)$ such that $W\subseteq V\subseteq \iota(U)$, and $W$ is $G_{\iota(x)}$-contractible in $V$. Let $H:W\times I\rightarrow V$ be a $G_{\iota(x)}$-contraction, $H_0$ is the inclusion map and $H_1:W\ni w\mapsto v_x\in V,$ for some $v_x\in V$ a $G_{\iota(x)}$-fixed point. Observe that $P=\iota^{-1}(W)\subseteq U$ is a $G_x$-invariant neighborhood of $x$ in $X$. Define $$F:P\times I\ni (p,t)\mapsto \iota^{-1}(H(\iota(p),t))\in  U.$$ If $p\in P,$  we have
$F_0(p)=\iota^{-1}(H(\iota(p),0))=p$, and $F_1(p)=\iota^{-1}(H(\iota(p),1))=\iota^{-1}(v_x)$,        
   which say that $F_0$ is the inclusion map and $F_1$ is  constant with value the $G_x$-fixed point $\iota^{-1}(v_x)\in U$. It only remains to verify that $F$ is a $ G_x$-map.
    Let $t\in I$, $g\in G_x$ and $p\in G$ so that $(g,x)\in G*X$. We will show that $F_t(\theta_g(p))=\theta_g(F_t(p))$. Pick $m,n\in U$ so that $H_t(\iota(p))=\iota(n)=[1,n]$ and $\mu_g(H_t(\iota(p)))=\iota(m)=[1,m]$, then $(g,n)\in G*X,$ $\theta_g(n)=m,$ and
    \begin{center}
        $F_t(\theta_g(p))=\iota^{-1}(H_t(\iota(\theta_g(p)))=\iota^{-1}(H_t(\mu_g(\iota(p)))=\iota^{-1}(\mu_g(H_t(\iota(p)))=m$,
    \end{center}
    and 
     \begin{center}
        $\theta_g(F_t(p))=\theta_g(\iota^{-1}(H_t(\iota(p)))=\theta_g(n)=m$,
    \end{center}
    this ends the proof.
    \end{proof}


\subsection{Globalization of partial actions and ANE's}\label{gane}
Let $X$ be a metrizable space equipped with a global action of a topological group $G$. Recall that $X$ is called an {\it equivariant absolute neighborhood extensor} (resp. an {\it equivariant absolute extensor}) denoted by $G$-ANE (resp. by $G$-AE) provided that for every metrizable space $Y$ equipped with a global action of $G$ and for every closed $G$-subset $A$ of $Y$, every $G$-map $f:A\rightarrow X$ extends to a $G$-map $F: U\rightarrow X$, for some invariant neighborhood $U$ of $A$ in $Y$ (or $U=Y$). We refer to \cite{A2}, \cite{A3}, \cite{A} for the equivariant theory of retracts.  When $G$ is  trivial, in the previous definitions we obtain the classical notion of an ANE (resp. an AE), (see \cite{SH}).
Given a metrizable $G$-space $(X,\theta)$, we   shall  establish sufficient conditions for the enveloping space $X_G$ to be an ANE or a $G$-ANE space. A characterization using homotopic properties is presented in \cite[Theorem 4.4]{AAS}.
We recall  another very important characterization below.
\begin{teo}\rm{\cite[Theorem 1]{YS} \cite[Theorem 4.2]{JJ2} \cite[Theorem 1.6]{AAMV} 
	Let $G$ be a compact Lie group and $(X,\theta)$ be a metrizable $G$-space with a finite structure, where $\theta$ is global. Then $X$ is a $G$-ANE (resp. a $G$-AE), if and only if, $X[K]$ is an ANE (resp. an AE) for each closed subgroup $K$ of $G$.}
\end{teo}

\begin{ejem}\cite[Example 3.4]{MPR}. Consider the partial action $\theta= (X_n,\theta_n)_{n\in \mathbb{Z}_4}$, of the discrete group $\mathbb{Z}_4$ on $X={S}^1$,  given by  $\theta_0={\rm Id}_{{  S}^1}$; $\theta_1:X_3\to X_1$ by $\theta_1(e^{it})=e^{i(t+\pi)}$;  $\theta_3=\theta_1^{-1}$, $\theta_2={\rm Id}_{X_2},$ where the sets $X_i$ are illustrated as follows:
	
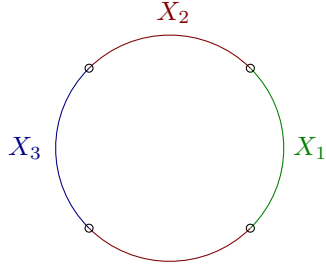
\begin{center}
		\begin{tikzpicture}[scale=1.5]
			\draw[green!50!black] (0.7071,-0.7071) arc (-45:45:1) (1,0) node[anchor=west]{$X_1$};
			\draw[blue!50!black] (-0.7071,0.7071) arc (135:225:1) (-1.5,0) node[anchor=west]{$X_3$};
			\draw[red!50!black] (0.7071,0.7071) arc (45:135:1)  (-0.7071,-0.7071) arc (225:315:1) (-0.2,1.2) node[anchor=west]{$X_2$};
			\draw (0.7071,-0.7071) circle (1pt) (0.7071,0.7071) circle (1pt) (-0.7071,0.7071) circle (1pt) (-0.7071,-0.7071) circle (1pt);
		\end{tikzpicture}
	\end{center}
Let $H=\{0,2\}$ and $K=\{0\}$ be subgroups of $\mathbb{Z}_4$. Observe that $X[H]=X_2$ is open (then is an ANE), and $X[K]=X$ is an ANE.
\end{ejem}

We want to use the characterization of $G$-ANE's established in \cite[Proposition 3.2]{JJ}. For this, one needs to relate $\dim X$ and $\dim X_G$  when these dimensions make sense. Motivated by Example \ref{matriz} where $\rm{dim}\ $$X=\rm{dim}\ $$X_G=1$, we have the following lemma.

\begin{lem}\label{dimen}   

	Let $(X,\theta)$ be a metrizable separable $G$-space whose enveloping space $X_G$ is metrizable and separable. If $\theta$ is nice, then $\dim X=\dim X_G$.
\end{lem}
\begin{proof}	 We only need  to show that $\dim X_G\leq \dim X$. Suppose $\dim X<\infty$. Since $X_G$ is Lindelöf, there exist a   subset $\{g_n: n\in \mathbb{N}\}$ of $G$ such  that $X_G=\bigcup\limits_{i=1}^{\infty} \mu_{g_i}(\iota(X))$. Observe that $\mu_g(\iota(X))$ is $F_\sigma$ in $X_G$ for each $g\in G$. Then we are done.	\end{proof}


\begin{teo}\label{comp}{\rm
 Let $G$ be a compact Lie group and $(X,\theta)$ be a finite-dimensional separable metrizable $G$-space whose enveloping space $X_G$ is metrizable. If $\theta$ is nice, then the following statements are equivalent:
 \begin{enumerate}
     \item [(i)] $X$ is locally $G$-contractible;
     \item [(ii)] $X_G$ is a $G$-ANE.
 \end{enumerate}}
\end{teo}
\begin{proof} Notice that $X_G$ is separable  because $X$ and $G$ are separable and the quotient map
 $p^\theta: G\times X\to X_G$ defined in \eqref{qo} is open. For (i) $\Rightarrow$ (ii), it follows from Theorem \ref{locallyc} that $X_G$ is locally $G$-contractible. Since $\dim X_G<\infty$ by Lemma \ref{dimen}, we obtain from \cite[Proposition 3.2]{JJ} that $X_G$ is a $G$-ANE.

    (ii) $\Rightarrow$ (i) The space $X_G$ is locally $G$-contractible by \cite[Proposition 3.2]{JJ}, then the fact that $X$ is locally $G$-contractible follows by  Theorem \ref{locallyc}.
    \end{proof}

The following result is clear but worth highlighting it,  because as we saw in Example \ref{matriz}, $X_G$ is not an AE even though $X$ is.

\begin{pro}\label{ANE}{\rm
	Let $(X,\theta)$ be a metrizable separable $G$-space whose enveloping space $X_G$ is metrizable. If $\theta$ is nice, then $X$ is an ANE if and only if $X_G$ is an ANE. }
\end{pro}
\begin{proof} Suppose that $X$ is an ANE. Since $X_G=\bigcup\limits_{g\in G}\mu_g(\iota(X))$, and $\mu_g(\iota(X))$ is an ANE for each $g\in G$, then the result follows by \cite[Chapter 2, Theorem 17.1]{SH}. For the converse, we assume that $X_G$ is an ANE. But by (iii) of Proposition \ref{Kellen1}  the space $X$ is homeomorphic to an open subset of $X_G$, then the assertion follows from \cite[Chapter 2, Proposition 6.1]{SH}.
\end{proof}

\begin{coro}\label{cl}
        Let $G$ be a compact Lie group and $(X,\theta)$ be a metrizable $G$-space whose enveloping space $X_G$ is metrizable. If $\theta$ is nice and free, then $X$ is an ANE if and only if $X_G$ is a $G$-ANE.
\end{coro}
\begin{proof}  It is not difficult to verify that the enveloping action $\mu$ of $\theta$ is free. Suppose that $X$ is an ANE, then  $X_G$ is an ANE thanks to Proposition \ref{ANE} and the result follows by \cite[Theorem 3.2]{AAS}. For the converse, observe that $X_G$ is an ANE, then $X$ is an ANE by Proposition \ref{ANE}.
\end{proof}



 To prove Theorem \ref{automatic} we need the next lemma.
 
 \begin{lem}\label{finita} The following statement holds.
 \begin{enumerate}
     \item [(i)]  Let $Y$ be a metric space, $n\in \mathbb{N}$ and $Y_1,\dots,Y_n$ be closed subsets of $Y$ such that  $Y=\bigcup\limits_{i=1}^nY_i$. If $\bigcap_{j\in J}Y_j$ is an ANE each subset $J$ of $\{1,2,\dots,n\}$, then $Y$ is an ANE.
     
     \item [(ii)]  Let $(X,\theta)$ be a $G$-space and let $\mu: G\times X_G\rightarrow X_G$ be its enveloping action. If $H$ is a subgroup of $G$, then:
	\begin{center}
		$X_G[H]=\bigcup\limits_{g\in G} \mu_g(\iota(X)[g^{-1}Hg])$.
	\end{center}
 \end{enumerate}
 	
 \end{lem}
 \begin{proof}
 	(i) We proceed by induction on $n$. The case $n=2$ is treated in \cite[Chapter 2, Proposition 10.1]{SH}. Suppose that the assumption holds for  $n\geq 2.$ Write $ Y=\bigcup\limits_{i=1}^{n+1}Y_i$  as a union of closed subsets, such that  $\bigcap_{j\in J}Y_j$ is an ANE for each $J\subseteq \{1,2,\dots,n+1\}$. By the inductive hypothesis, $Z=\bigcup_{i=1}^nY_i$ is an ANE. Notice that $Z\cap Y_{n+1}=\bigcup_{i=1}^n (Y_i\cap Y_{n+1})$ is an ANE. Indeed, take $J\subseteq\{1,2,\dots,n\}$. Then
 	\begin{center}
 		$\bigcap\limits_{j\in J} (Y_j\cap Y_{n+1})=\bigcap\limits_{j\in J\cup\{n+1\}}Y_j$   \,\, is an ANE,
 	\end{center}
 and by the inductive hypothesis $Z\cap Y_{n+1}$ is an ANE. Thus $Y=Z\cup Y_{n+1}$ is an ANE  thanks to \cite[Chapter 2, Proposition 10.1]{SH}.\vspace{0.2cm}

 \noindent (ii)
	Take $z\in X_G[H]$ and $g\in G$, $x\in X$ with $z=[g,x]$. Since $ G_z=G_{\mu_g(\iota(x))}=gG_{\iota(x)}g^{-1}$ and $H\subseteq G_z$, then  $g^{-1}Hg\subseteq G_{\iota(x)}$ and $\iota(x)\in \iota(X)[g^{-1}Hg]$. This show that $z\in \mu_g(\iota(X)[g^{-1}Hg])$.
	Conversely, take  $g\in G$, $x\in X$ such that $\iota(x)\in i(X)[g^{-1}Hg]$ and let $z=\mu_g(\iota(x))$. Then $g^{-1}Hg\subseteq G_{\iota(x)}$ and $H\subseteq gG_{\iota(x)}g^{-1}=G_{\mu_g(x)}=G_z$. Thus $z\in X_G[H]$.\end{proof}


\begin{teo}\label{automatic}{\rm
	Let $G$ be a finite group and $(X,\theta)$ be a  finitely-dimensional separable metric $G$-space such that $\theta$ is nice and $G*X$ is closed. If  $X[H]$ is an ANE for each subgroup $H$ of $G$ then $X_G$ is a $G$-ANE.}
\end{teo}
\begin{proof}
	
	Since $G$ and $X$ are separable, then $X_G$ is a separable space. Moreover,  the fact that $G*X$ is closed implies that $X_G$ is metrizable thanks to \cite[Theorem 3.2]{MPV1}.  Now,
let $H$ be a subgroup of $G$ and observe that $X_G[H]$ is an ANE. Indeed, since  $\iota(X[K])=\iota(X)[K]$ for each subgroup $K$ of $G$ and the map $\iota$ is an embedding we have that $\iota(X)[K]$ is an ANE. In particular, $\iota(X)[g^{-1}Hg]$ is an ANE for each $g\in G$. Thus follows by (ii) of  Lemma \ref{finita} it follows that $X_G[H]$ is a finite union of ANE closed subspaces. In addition, observe that $$\bigcap_{i=1}^m\iota(X)[K_i]=\iota(X)\left[\left\langle \bigcup\limits_{i=1}^mK_i\right\rangle\right]$$
	 is an ANE for every subgroups $K_1,\dots,K_m$ of $G$, where $\left\langle \bigcup\limits_{i=1}^mK_i\right\rangle$ denotes the subgroup generated by $\bigcup\limits_{i=1}^mK_i$. Since $G$ is finite, then  (i) in Lemma  \ref{finita} implies that $X_G[H]$ is an ANE. Moreover, it follows by Lemma \ref{dimen}  that $\rm{dim}\ X_G<\infty$, then using \cite[Example 1.4]{JJ2} we can apply \cite[Theorem 4.2]{JJ2} to conclude that $X_G$ is a $G$-ANE, as desired.
 \end{proof}
	

\begin{ejem}
	Let $n\in \mathbb{Z}^+$ and $\{U_k: k\in\{1,2,\cdots,n\}\}$ be  a collection of disjoint nonempty open subsets of $\mathbb{R}$. Observe that $X=\bigcup\limits_{i=1}^n U_i$ is a ANE. Consider $\mathbb{Z}_2=\{0,1\}$ acting partially on $X$ under $\theta: \mathbb{Z}_2*X\rightarrow X$, where $X_1=U_1$, $X_0=X$, $\theta_0={\rm id}_X$ and $\theta_1= {\rm id}_{U_1}$. It is clear that the hypothesis of Proposition \ref{automatic} is satisfied. Then $X_{\mathbb{Z}_2}$ is a $\mathbb{Z}_2$-ANE.

\end{ejem}

	\end{document}